\documentclass[a4paper,12pt]{article}
\usepackage{graphicx} 

\usepackage{tikz}
\newcommand*\circled[1]{\tikz[baseline=(char.base)]{
            \node[shape=circle,draw,inner sep=2pt] (char) {#1};}}

\usepackage{amsmath,amsthm,amscd,amssymb,eucal,mathrsfs}
\usepackage{amsfonts}
\usepackage{latexsym}
\usepackage{fullpage}
\usepackage{dsfont}
\usepackage[all]{xy}

\newtheorem{thm}{Theorem}[section]

\newtheorem{cor}[thm]{Corollary}
\newtheorem{remark}[thm]{Remark}
\newtheorem{definition}[thm]{Definition}

\newtheorem{example}[thm]{Example}

\newtheorem{Proposition}[thm]{Proposition}

\newtheorem{Lemma}[thm]{Lemma}

\newtheorem{Cor}[thm]{Corollary}

\newcommand{\mathset}[1]{{\left\{#1\right\}}}
\newcommand{\absolute}[1]{\left\lvert#1\right\rvert}
\newcommand{\norm}[1]{\left\|#1\right\|}

\DeclareMathOperator{\id}{id}

\DeclareMathOperator{\supp}{supp}

\DeclareMathOperator{\Range}{Ran}

\title{Diffusion operators on $p$-adic analytic manifolds}
\author{Patrick Erik Bradley
\\
Institute of Photogrammetry and Remote Sensing (IPF)
\\
Geodetic Institute Karlsruhe (GIK)
\\
Karlsruhe Institute of Technology (KIT)
\\
Email: bradley@kit.edu}
\date{\today}

\begin{document}

\maketitle

\begin{abstract}
Kernel functions for 
Laplacian integral operators are constructed on $p$-adic analytic manifolds using charts and transition maps from an atlas with connected nerve complex. In the compact case, an operator of Vladimirov-Taibleson type  parametrised by a real parameter $s$ is defined. Its kernel function uses a geodetic-like distance function on the nerve complex of its atlas.
The $L^2$-spectrum of this operator is established, and it is shown that it gives rise to a Feller semigroup. In this way, the Cauchy problem for the corresponding heat equation is solved in the positive by a transition function of a Markov process. The existence of a heat kernel function and a Green function in the case $s>1$ is proven. As an application, it is shown how to express the number of points on the reduction curve defined over the residue field of an elliptic curve with good reduction in terms of the eigenvalues of a Vladimirov-Taibleson-like operator. This provides for an alternative way of counting points on elliptic curves defined over finite fields.
\end{abstract}

\tableofcontents
\section{Introduction}

There  is meanwhile a lot of literature on diffusion on $p$-adic domains in a variety of different guises, the starting point being given by the authors of \cite{Taibleson1975,VVZ1994}. Many authors are concerned with the space $\mathds{Q}_p^n$ for $n\ge 1$, where $\mathds{Q}_p$ stands for the $p$-adic number field, e.g.\ \cite{RW2023}. A deep result in this context is that $p$-adic Brownian motion is a scaling limit \cite{Weisbart2024}. One also finds literature on diffusion on compact subdomains of $K^n$, where $K$ is a non-archimedean local field, e.g.\ \cite{Kochubei2018,PW2024,ZunigaNetworks}. Other non-archimedean spaces on which diffusion is studied are vector spaces over local fields \cite{PRSWY2024}, the finite ad\`eles \cite{Urban2022}, Mumford curves \cite{brad_thetaDiffusionTateCurve,SchottkyDiff} and their finite quotients \cite{brad_HeatMumf}. 
\newline

Diffusion operators are even studied on non-commutative $p$-adic Lie groups through their unitary representation theory \cite{EngelGroup,UnitaryDual_p}. More general ultrametric spaces are also studied in this context with a plethora of diffusion Laplacians or sub-Laplacians being used \cite{BGPW2014,VPzeta}. The former reference inspired the construction of directional Laplacians on $p$-adic Lie groups \cite{RodriguezDiss}, and lead to $p$-adic weak imitations of classical partial differential operators in order to study boundary value problems on compact $p$-adic domains  \cite{ellipticBVP}. The reference \cite{VPzeta} extended ideas from \cite{PB2009} and found applications on the transcendent $p$-adic points of Shimura curves \cite{SchottkyTrans_p}. The study of time-dependent $p$-adic diffusion is
initiated in \cite{nonAutonomousDiffusion,Ledezma_Energy}. Recent applications of diffusion operators consist in the extraction of topological and geometric information from $p$-adic domains \cite{BL_shapes_p,brad_thetaDiffusionTateCurve,HearingGenusMumf}.
\newline

Diffusion is mostly understood through a Markov process. Over $p$-adic fields, these are also studied \cite{Zuniga2008}. More general $p$-adic stochastic processes are in the focus of \cite{Zuniga2017}. Their relation to $p$-adic pseudodifferential equations is extensively laid out in the recent book \cite{Zuniga2025}.
\newline

All this having been said, defining a
diffusion operator on $p$-adic analytic manifolds in the sense of \cite{Serre1964,Schneider2011} using charts and transition functions is still lacking. The scope of this article is to fill this gap, and to study the case of $n$-dimensional compact $p$-adic analytic manifolds, since they are always endowed with a nowhere vanishing analytic differential $n$-form according to \cite{Serre1965}. This is an asset, because from \cite{WeilAAG,IgusaLocalZeta} one is informed that  such a differential form naturally defines a measure on the manifold.
For example, elliptic curves have nowhere vanishing invariant differential $1$-forms, and such can be used also in the case of an elliptic curve $E$ with good reduction to provide for a measure on its $K$-rational points $E(K)$, where $K$ is a $p$-adic local field. This approach allows to extend the construction of diffusion processes of Tate elliptic curves of \cite{brad_thetaDiffusionTateCurve} via Laplacian integral operators in an entirely different manner to the case of elliptic curves $E$ over $K$ with good reduction. As an application, the number of $\mathds{F}_q$-rational points on the reduction curve $\overline{E}$
can be expressed in terms of the spectrum of a Laplacian integral operator on $E(K)$.
\newline

The following section recalls the necessary constructions for $p$-adic analytic manifolds and how to obtain a measure as alluded to in the previous paragraph. Section 3 studies charts and transition maps from a certain kind of atlantes having connected nerve complex and the property that its transition functions on overlaps take balls to balls of equal radius. These are used in order to obtain well-defined kernel functions on a $p$-adic analytic manifold $X$, as well as corresponding Laplacian integral operators on $X$. This allows to define an analogue $\Delta^s$ of the Vladimirov-Taibleson operator on compact $p$-adic analytic manifolds depending on a real parameter $s$, and whose kernel function depends on the geodetic distance between simplices of the nerve complex associated with the atlas. The remainder of this section is dedicated to
finding their $L^2$-spectra, showing that such operators define a Feller semigroup, to solving the  Cauchy problem for the corresponding heat equation, and to proving the existence of heat kernel functions and Green functions in the case $s>1$.
The last Section 4 makes these constructions explicit in the case of an elliptic curve with good reduction, and then is able to count the points on the reduction curve defined over the residue field from the infimum of the wavelet eigenvalues of $\Delta^s$ with varying degree of precision, depending on $s\in\mathds{R}$.
\newline

Lastly, to fix notation, let $(K,\absolute{\cdot})$ be a non-archimedean local field. The maximum norm $\norm{\cdot}_K$ on $K^n$ is also known to define an ultrametric.
The given $p$-adic analytic manifold will be denoted as $X$. Since 
\[
dx=dx_1\wedge\dots\wedge dx_n
\]
already denotes a differential $n$-form on $K^n$, the notation
\[
\absolute{dx}=\absolute{dx_1}\wedge\dots\wedge\absolute{dx_n}
\]
will be used for indication the Haar measure on $K^n$. However, the individual coordinate Haar measures $\absolute{dx_i}$ will not be used in this article.

\section{$p$-adic analytic manifolds}

The notion of $p$-adic analytic manifold $X$ of dimension $n$ over a non-archimedean local field $K$ can be found e.g.\ in the books \cite{Serre1964,Schneider2011}. We recall some notions here.

\subsection{Atlantes and Differential Forms}

Let $X$ be a Hausdorff space.
An $n$-chart on $X$ is a pair
$(U_\alpha,\phi_\alpha)$, where $U_\alpha\subseteq X$ is open,
\[
\phi_\alpha\colon U_\alpha\to K^n
\]
is a  map such that $\phi_\alpha(U_\alpha)$ is open in $K^n$, and $\phi_\alpha$ is a  homeomorphism onto its image. Further, two charts  $(U_\alpha,\phi_\alpha)$, $(U_\beta,\phi_\beta)$ are \emph{compatible}, if  there is a transition map $\tau_{\alpha\beta}$ fitting into the following commutative diagram
\[
\xymatrix{
&U_\alpha\cap U_\beta\ar[dl]_{\phi_\alpha}\ar[dr]^{\phi_\beta}
\\
\phi_\alpha(U_\alpha\cap U_\beta)\ar[rr]_{\tau_{\alpha\beta}}&&\phi_\beta(U_\alpha\cap U_\beta)\,,
}
\]
and this transition map $\tau_{\alpha\beta}$ is bi-holomorphic whose inverse is the transition map $\tau_{\beta\alpha}=\tau_{\alpha\beta}^{-1}$ in the reverse horizontal direction in that diagram.
An \emph{atlas} of $X$ is a family 
\[
\mathcal{A}=\mathset{(U_\alpha,\phi_\alpha)\mid\alpha\in I}
\]
of pairwise compatible charts such that the sets $U_\alpha$ with $\alpha\in I$ form an open covering of $X$. Two atlantes $\mathcal{A}$, $\mathcal{B}$ of $X$ are equivalent if $\mathcal{A}\cup\mathcal{B}$ is also an atlas of $X$. An equivalence class of atlantes on $X$ is called a \emph{$p$-adic analytic structure} on $X$. Each analytic structure on $X$ possesses an atlas which is maximal in its equivalence class with respect to inclusion of sets, cf.\ e.g.\ \cite[Remark 7.2]{Schneider2011}. Finally, a $p$-adic analytic manifold is a pair $(X,\mathcal{A})$ consisting of a Hausdorff space $X$ and an atlas $\mathcal{A}$ of $X$ in the above sense.
\newline

In \cite[Chapter 2.2]{WeilAAG} or \cite[Chapter 7.4]{IgusaLocalZeta}, one can learn how an analytic differential $n$-form $\omega$ on a $p$-adic analytic manifold $X$ 
gives rise to a measure $\mu_X=\absolute{\omega}$ on $X\setminus V(\omega)$, where $V(\omega)\subset X$ is the vanishing set of $\omega$, in the following way: express locally in a chart $(U_\alpha,\phi_\alpha)$ the differential form $\omega$ as
\[
\omega_\alpha=f_\alpha\,dx
\]
on $U_\alpha$ with $f_\alpha\colon U_\alpha\to K^n$ an analytic map, and where 
\[
dx=dx_1\wedge\dots\wedge dx_n
\]
is the standard differential $n$-form on $K^n$. More precisely, $\omega$ is given on $\phi_\alpha(U_\alpha)$ as
\[
\omega_\alpha(\phi_\alpha(x))=f_\alpha(\phi_\alpha(x))\,dx
\]
with the chart map $\phi_\alpha\colon U_\alpha\to K^n$. Then, outside of the vanishing set of $\omega_\alpha$, the measure $\absolute{\omega}$ is defined on $(U_\alpha,\phi_\alpha)$ as
\[
\mu_X(A)=\int_{\phi_\alpha(A)}
\absolute{f_\alpha(\phi_\alpha(x))}\absolute{dx}
\]
for a set $A\subset U_\alpha$ such that 
$\phi_\alpha(A)$ is measurable w.r.t.\ the
 Haar measure on $K^n$ which has been denoted as $\absolute{dx}$ in order to not confuse it with the differential form $dx$ on $K^n$. The measure $\mu_X(A)$ can also be written as
 \[
\mu_X(A)=\int_A\absolute{\omega}=
\int_X1_A(x)\absolute{\omega(x)}
=\int_X 1_A(x)\,d\mu_X(x)\,,
 \]
 and possibly in some other more or less intuitive ways. In a similar way, the same holds for functions $f\colon X\to\mathds{C}$. This, together with a way to define kernel functions, will be made explicit in later sections.

\subsection{Equalising Atlantes}

Given a covering $\mathcal{U}$ of a topological space, one can define the \emph{nerve graph} $N_1(\mathcal{U})$ of $\mathcal{U}$ as follows: its vertices are the elements of $\mathcal{U}$, and an edge is defined between two distinct elements $U,V\in\mathcal{U}$, if and only if their intersection $U\cap V$ is non-empty. The nerve graph is in fact the $1$-skeleton of the nerve complex $N(\mathcal{U})$ used for constructing \v{C}ech complexes.
\newline

Given  a $p$-adic analytic manifold $(X,\mathcal{A})$, define the nerve complex $N(\mathcal{A})$ of the atlas $\mathcal{A}$ as the nerve complex of the underlying covering of $X$ given by the charts in $\mathcal{A}$. The nerve graph $N_1(\mathcal{A})$  is defined as the $1$-skeleton of $N(\mathcal{A})$.

\begin{Lemma}
Any compact $p$-adic analytic manifold $(X,\mathcal{A})$ has an  atlas equivalent to $\mathcal{A}$ whose nerve graph is a finite trivial graph, i.e.\ a finite graph without edges. 
\end{Lemma}

\begin{proof}
This is the content of Serre's theorem \cite[Th\'eor\`eme (1)]{Serre1965}. 
\end{proof}

This result shows that, in principle the number of connected components of a nerve graph represented by an atlas of a given analytic structure on a compact manifold can be arbitrary.

\begin{definition}
A $p$-adic analytic structure on a Hausdorff space $X$ is called \emph{connected}, if it contains an atlas whose nerve graph is connected.
\end{definition}

\begin{definition}
Let $F\colon U \to V$ be a locally bi-analytic map between open subsets $U,V$ of $K^n$. The smallest $N \in\mathds{Z}$ such that $F$ takes balls of radius $p^{-N}$ to balls of
the same radius, is called the \emph{equalising number} of $F$. If no such integer exists, then the
equalising number will be defined as $-\infty$. The notation is
\[
e(F) = \text{the equalising number of $F$}
\]
for a given such map $F$. In the case that $e(F)$ is finite, the map $F$ is called equalising.
\end{definition}

\begin{example} An example of a bi-analytic map which is not equalising is given by
\[
f\colon K \to K,\; x \mapsto \pi x\,,
\]
as even no disc in $K$ is mapped to a disc having the same radius.
\end{example}

\begin{definition}
The \emph{equalising number} of an atlas $\mathcal{A}$ of a $p$-adic manifold is 
\[
e(\mathcal{A})=\sup\mathset{e(\tau_{\alpha,\beta})\mid\tau_{\alpha\beta}\,\text{is the transition map for $(U_\alpha,\phi_\alpha),(U_\beta,\phi_\beta)\in\mathcal{A}$}}\,,
\]
and is an integer, $\infty$, or $-\infty$. In the case that $e(\mathcal{A})\neq-\infty$,
 the atlas is called \emph{equalising}, and $e(\mathcal{A})$ the \emph{equalising number} of $\mathcal{A}$.
\end{definition}

\begin{thm}\label{finiteEqAtlas}
An analytic structure of a compact $p$-adic manifold contains a finite equalising atlas with a finite equalising number. 
\end{thm}

\begin{proof}
By compactness, the analytic structure contains a finite atlas $\mathcal{A}$. Hence, if all transition maps are equalising, then the quantity $e(\mathcal{A})$ is a finite integer.

\smallskip
Let now $(U,\phi),(V,\psi)$ be a pair of compatible charts in $\mathcal{A}$ with  transition map 
\[
F\colon \phi(U\cap V)\to\psi(U\cap V)\,,
\]
where $F(x)$ is w.l.o.g.\ the bi-analytic function
\[
F(x)=a_0+Ax+\text{h.o.t.}
\]
on $K^n$. Assuming that
\[
\norm{A-I}\ge 1\,,
\]
the following equalising method will be applied: First,
there exists $m\in\mathds{Z}$ such that
\[
H(x)=p^mF(x)+x=p^ma_0+\left(p^mA+I\right)x+\text{h.o.t.}
\]
satisfies
\[
\norm{A'-I}<1
\]
for 
\[
A'=p^mA+I\,.
\]
Then the diagram
\[
\xymatrix@C=4pt{
&&U\cap V\ar[dl]_\phi\ar[dr]^\psi
\ar@/_4pc/[ddl]_{\phi'}\ar@/^4pc/[ddr]^{\psi'}
\\
&K^n\ar[rr]^F\ar[d]_{\tilde\phi}&&K^n\ar[d]^{\tilde\psi}
\\
&K^n\ar[rr]_H&&K^n
}
\]
commutes, where the other maps are given as
\begin{align*}
\tilde\phi&=\id\,,
\\
\phi'&=\phi\,,
\\
\tilde\psi(z)&=p^mz+F^{-1}(z)\,,
\\
\psi'&=\tilde\psi\circ\psi\,.
\end{align*}
It follows that the charts
\[
(U,\phi),(U,\phi'),(V,\psi),(V,\psi')
\]
are compatible, and the pair $(U,\phi'),(V,\psi')$ is an equalising pair of charts which can replace $(U,\phi),(V,\psi)$, and the new transition map is given by
\[
\phi'(U\cap V)\to\psi'(U\cap V),\;
x\mapsto \tilde\psi(F(x))=H(x)\,,
\]
because it does indeed hold true that
\[
\tilde\psi(F(x))=p^mF(x)+F^{-1}(F(x))=H(x)
\]
for $x\in\phi'(U\cap V)=\phi(U\cap V)$.

\smallskip
Now, given again atlas $\mathcal{A}$, assume w.l.o.g.\
that 
\begin{align}\label{chartTriple}
(U,\phi_U),(V,\phi_V),(W,\phi_W)
\end{align}
are pairwise overlapping charts in $\mathcal{A}$ for which every pair is not equalising.
The only interesting case  is that the triple intersection $U\cap V\cap W$ is non-empty. Then the
three transition maps form a diagram
\[
\xymatrix{
\phi_U(U\cap V\cap W)\ar[rr]^{\tau_{\phi_U\phi_V}}\ar[dr]_{\tau_{\phi_U\phi_W}}&&
\phi_V(U\cap V\cap W)\ar[dl]^{\quad\tau_{\phi_V\phi_W}}
\\
&\phi_W(U\cap V\cap W)
}
\]
forming a part of the larger diagram
\begin{align}\label{largerDiagram}
\xymatrix@=40pt{
K^n\ar[rr]^{F_{UV}}\ar[ddr]_{F_{UW}}&&K^n
\\
&U\cap V\cap W\ar[ul]_{\phi_U}\ar[ur]^{\phi_V}\ar[d]^{\phi_W}
\\
&K^n\ar[uur]_{F_{WV}}
}
\end{align}
The question is, if it is possible to obtain an equalisation of the triple (\ref{chartTriple}) fitting into the diagram (\ref{largerDiagram}).
The equalising method above allows to keep $(U,\phi_U)$ and simultaneously modify $(V,\phi_V)$ and $(W,\phi_W)$ such that
we may assume that $F_{UV}$ and $F_{UW}$ are both equalising bi-analytic maps. This means that one can now equalise $F_{WV}$ as follows:
\[
\xymatrix{
K^n\ar[r]^{F_{WV}}\ar[d]^{\id}&K^n\ar[d]_{\tilde\psi}
\\
K^n\ar[r]_H&K^n
}
\]
such that $H$ is an equalising bi-analytic map. 

\paragraph{Claim.} The map $\tilde\psi\circ F_{UV}$ is also equalising.

\begin{proof}[Proof of Claim]
It holds true that
\[
\tilde\psi(F_{UV}(x))
=p^mF_{UV}(x)+F_{UV}^{-1}(F_{WV}(x))
=p^mF_{UV}(x)+F_{UW}(x)\,.
\]
Hence, the linear term of this map is
\[
(p^mA_{UV}+A_{UW})x\,,
\]
where $A_{UV}$ and $A_{UW}$ are the matrices appearing in the first-order approximations of $F_{UV}$ and $F_{UW}$, respectively.
Thus, the matrix
\[
B=p^mA_{UV}+A_{UW}
\]
satisfies the inequality
\[
\norm{B-I}=\norm{(A_{UW}-I)+p^mA_{UV}}
<1
\]
for $m>>0$, because $\norm{A_{UW}-I}<1$ by assumption.
\end{proof}

By induction, it now readily follows that in the case of a non-empty intersection
\[
U_0\cap \dots\cap U_r\neq\emptyset
\]
with $r\in\mathds{N}$, the above proof can be extended to show that all pairs of transition maps for the coverings $(U_1,\phi_1),\dots,(U_r,\phi_r)$ can be made equalising.
\end{proof}

\subsection{A weighted nerve complex}\label{sec:nerveComplex}

The nerve complex $N(\mathcal{U})$ of a covering $\mathcal{U}=\mathset{U_i\mid i\in I}$ of a topological space, with $I$ assumed at most countable, is defined as follows:
\begin{align*}
U_i&\neq\emptyset,\quad i\in I&\text{(vertices)}
\\
U_{ij}=U_i\cap U_j&\neq\emptyset,\quad i\neq j\in I&\text{(edges)} 
\\
&\vdots&\vdots\quad
\\
U_{i_0\dots i_k}=U_{i_0}\cap\dots\cap U_{i_k}&\neq\emptyset,\quad i_0,\dots,i_k\in I\;\text{pairwise distinct}&\text{($k$-simplices)}
\\
&
\vdots &\text{(etc.)}
\end{align*}
Of course,  $U_i\in\mathcal{U}$ is assumed non-empty, as it is
otherwise tacitly removed from $\mathcal{U}$.
A $k$-simplex $U_{J_k}$ is a $k$-facet of an $\ell$-simplex $U_{J_\ell}$ with $k<\ell$ and $J_k,J_\ell\subseteq I$ such that $\absolute{J_k}=k+1$, $\absolute{J_\ell}=\ell+1$, iff $J_k\subset J_\ell$. Here, we used the notation
\[
U_J=\bigcap\limits_{j\in J}U_j
\]
for a finite subset $J$ of $I$. A $k$-facet of  a $k+1$-simplex $\sigma$ is called a face of $\sigma$. 
\newline

A tacit assumption about the atlas $\mathcal{A}$, apart from the other stated assumptions, is that if $(U,\phi),(V,\psi)\in\mathcal{A}$ are such that $\phi\neq\psi$, then also $U\neq V$. This is a motivation for us to use equalising atlantes on $p$-adic analytic manifolds.
\newline

Let  $N(\mathcal{A})$ be the nerve complex of a compact $p$-adic analytic manifold $(X,\mathcal{A})$, where the atlas $\mathcal{A}$ is assumed finite and equalising. Further, $N(\mathcal{A})$ is assumed connected, although for the considerations in this section, this assumption is not needed.
Its $k$-simplices are now going to be given weights $w(\sigma)$ as follows:
\newline

For a $k$-simplex $\sigma=U_{i_0\dots i_k}$, define
\[
w(\sigma)=\mu_X(U_{i_0\dots i_k})\,,
\]
and thus obtain a weight function on the set of facets of $N(\mathcal{A})$.

\begin{definition}
 The pair $(N(\mathcal{A}),w)$ is called the \emph{weighted nerve complex} of the $p$-adic manifold $(X,\mathcal{A})$.
\end{definition}

As an example, let $\mathds{P}^n(K)$ be the $p$-adic projective $n$-space. It is shown in \cite[Chapter 3.1]{IgusaLocalZeta} that this is a $p$-adic analytic manifold, endowed with 
the  atlas consisting of the $n+1$ open sets
\[
U_i=\mathset{[x_0:\dots:x_n]\in\mathds{P}^n(K)\mid\absolute{x_i}=1,\;\forall j\neq i\colon\absolute{x_j}\le 1}
\]
for $i=0,\dots,n$. The chart maps can be written as
\[
\phi_i\colon U_i\to K^n,\;[x_0:\dots:x_n]\mapsto
\left[\frac{x_0}{x_i}:\dots:1:\dots:\frac{x_n}{x_i}\right]=
\left(\frac{x_0}{x_i},\dots,\frac{x_n}{x_i}\right)\,,
\]
where in the first expression after the arrow $\mapsto$, the $1$ is in the $i$-th place.
\newline

Observe that the image of each $\phi_i$ is $O_K^n$. In order to keep track of indices, write the elements of $\phi_i(U_i)$ as
\[
z=\left(z_0,\dots,\check{z}_i,\dots,z_n\right)\in \phi_i(U_i)=O_K^n\,,
\]
i.e.\ as $n+1$-tuples with $i$-th entry left out. This is justified, as $\mathds{P}^n(K)$ is a quotient of $K^{n+1}\setminus\mathset{0}$.
The transition maps are now given as
\[
\tau_{ij}\colon\phi_i(U_i\cap U_j)\to\phi_j(U_i\cap U_j),\;\left(z_0,\dots,\check{z}_i,\dots,z_j,\dots,z_n\right)\mapsto
\left(\frac{z_0}{z_j},\dots,\frac{1}{z_j},\dots,\check{z}_j,\dots,\frac{z_n}{z_j}\right)
\]
for $i< j$, and $\tau_{ii}=\id$, for $i,j=0,\dots,n$. If $i>j$, then
\[
\tau_{ij}=\tau_{ji}^{-1}\,,
\]
where $\tau_{ji}$ is defined as above.
\newline

Observe that for $i<j$: 
\[
\phi_i(U_i\cap U_j)=O_K^{j-1}\times S_0\times O_K^{n-j}\,,
\]
i.e.\ the $j$-th coordinate ranges over the unit sphere
\[
S_0=\mathset{x\in K\mid\absolute{x}=1}\,,
\]
whereas the other coordinates range over the whole unit disc $O_K$.
\newline

The measure $\absolute{\omega}$ on $\mathds{P}^n(K)$ is
on each chart $U_i$ ($i=0,\dots,n$) given
by the Haar measure, and transitions
as follows: the differentials of the functions $\frac{x_k}{x_i}$ on $U_{ij}$ are
\[
d\!\left(\frac{x_k}{x_i}\right)=\frac{d(x_k/x_j)}{x_i/x_j}-\frac{x_k/x_j}{\left(x_i/x_j\right)^2}\,d\!\left(\frac{x_i}{x_j}\right)
\]
for $k\neq i,j$, and
\[
d\!\left(\frac{x_j}{x_i}\right)
=d\!\left(\frac{1}{x_i/x_j}\right)=-\frac{d(x_i/x_j)}{\left(x_i/x_j\right)^2}\,,
\]
and then taking the wedge product yields
\begin{align*}
d\!\left(\frac{x_0}{x_i}\right)&\wedge\dots\wedge\left(x_i/x_i\right)^{\vee}\wedge\dots \wedge d\!\left(\frac{x_n}{x_i}\right)
\\
&=\pm\left(\frac{x_j}{x_i}\right)^{n+1}
d\!\left(\frac{x_0}{x_j}\right)
\wedge\dots\wedge\left(x_j/x_j\right)^{\vee}\wedge\dots\wedge
d\!\left(\frac{x_n}{x_j}\right)\,,
\end{align*}
where the sign depends on, whether  or not a transposition is needed in order to bring the differentials into the natural ordering.
\newline

Hence, on each chart, it is
\[
\absolute{\omega|_{U_i}}=\absolute{dx_0}\wedge\dots\wedge\absolute{dx_i}^\vee\wedge\dots\wedge\absolute{dx_n}\,,
\]
and on the overlap $U_{ij}$, it transitions as
\[
\absolute{\omega|_{\phi_i(U_{ij}})}=\frac{\absolute{x_i}^{n+1}}{\absolute{x_j}^{n+1}}\absolute{dx_0}\wedge\dots\wedge\absolute{dx_j}^\vee\wedge\dots\wedge\absolute{dx_n}
=\absolute{\omega|_{\phi_j(U_{ij})}}\\,.
\]
Hence the local Haar measures transition by a constant factor $1$ on overlaps.
\newline

In any case, the transition maps are $p$-adic analytic and take balls to balls of equal measure. Hence, 
the atlas
\[
\mathcal{A}=\mathset{(U_i,\phi_i)\mid i=0,\dots,n}
\]
is equalising with
\[
e(\mathcal{A})=1
\]
as its equalising number. Figure \ref{NerveProjectivePlane} shows the weighted nerve complex associated with the $p$-adic projective plane $(\mathds{P}^2(K),\mathcal{A})$. It is a $2$-simplex with constant weight for every simplex dimension. It also shows that the $0$-faces are actually balls with radius larger than $q^{-e(\mathcal{A})}=q^{-1}$.

\begin{figure}[ht]
\[
\xymatrix{
&\circled{\scriptsize 1}\ar@{-}[dl]_{1-q^{-1}}\ar@{-}[dr]^{1-q^{-1}}&
\\
\circled{\scriptsize 1}\ar@{-}[rr]_{1-q^{-1}}\ar@{}[urr]_{\!\!\!\!\!\!\!\!\!(1-q^{-1})^2}&&\circled{\scriptsize 1}
}
\]
\caption{The weighted nerve complex associated with $\left(\mathds{P}^2(K),\mathcal{A}\right)$.}\label{NerveProjectivePlane}
\end{figure}

\section{Laplacian operators on compact $p$-adic manifolds}

In this section, $X$ denotes a compact $p$-adic analytic manifold with a given connected analytic structure 
containing an equalising atlas $\mathcal{A}$ whose nerve graph is also connected. Furthermore, each open in the covering of $X$ given by $\mathcal{A}$ is assumed compact.
If one starts in the proof of Theorem \ref{finiteEqAtlas} 
with a finite atlas having a connected nerve graph, then one obtains an equalising atlas of $X$ having the same property.
This property is also assumed.
\newline

Again, assume that $(X,\mathcal{A})$ has a non-vanishing analytic differential $n$-form $\omega$. This is justified by \cite[Th\'eor\`eme (2)]{Serre1965}.

\subsection{Function spaces on $p$-adic manifolds}

In this subsection, $X$ need not be compact, but nevertheless has an  analytic
structure containing an equalising atlas $\mathcal{A}$ whose nerve graph is also connected. Furthermore, it is assumed that a non-vanishing analytic differential $n$-form $\omega$ exists on $(X, \mathcal{A})$.
\newline

Then a function
$h\colon X\to\mathds{R}$ is defined by the following data: a family  of functions
\[
h_\alpha\colon \phi_\alpha(U_\alpha)\to\mathds{R}
\]
for each chart $(U_\alpha,\phi_\alpha)$  ($\alpha\in I$),
satisfying the following condition: given charts 
\[
(U_\alpha,\phi_\alpha),(U_\beta,\phi_\beta)\in\mathcal{A}\]
such that
\[
U_{\alpha\beta}=U_\alpha\cap U_\beta\neq\emptyset\,,
\]
then the following diagram is commutative:
\begin{align}\label{diagram}
\xymatrix{
&U_{\alpha\beta}\ar[dl]_{\phi_\alpha}\ar[dr]^{\phi_\beta}&
\\
\phi_\alpha(U_{\alpha\beta})\ar[rr]^{\tau_{\alpha\beta}}\ar[dr]_{h_\alpha}&&\phi_\beta(U_{\alpha\beta})\ar[dl]^{h_\beta}
\\
&\mathds{R}
}
\end{align}
where $\tau_{\alpha\beta}$ is the  transition map associated with the charts $(U_\alpha,\phi_\alpha),(U_\beta,\phi_\beta)$ from the atlas.
\newline

Since indicator functions $1_A$ of measurable subsets $A\subseteq X$ can be defined in the same manner, the following definition is motivated by the equalising character of the atlas $\mathcal{A}$:

\begin{definition}
A set $B\subseteq X$ is a \emph{ball}, if for each chart $(U_\alpha,\phi_\alpha)\in\mathcal{A}$, $\phi(B)$ is a ball in $K^n$, and for any pair $(U_\alpha,\phi_\alpha),(U_\beta,\phi_\beta)\in\mathcal{A}$ with
$B\subseteq U_\alpha\cap U_\beta$, the transition map $\tau_{\alpha\beta}$ takes $\phi_\alpha(B)$ to $\phi_\beta(B)$ which is a ball in $K^n$
of equal radius. This radius is defined as the radius of a given ball $B\subseteq X$.
\end{definition}

Notice that in a $p$-adic manifold having an equalising atlas, the notion of radius of a ball is well-defined.

\begin{definition}
The \emph{join} of two points $x,y\in X$ of $X$ is 
\[
x\wedge y=\bigcap\limits_{B}B\,,
\]
where $B\subseteq X$ runs through all balls of $X$ containing $x,y\in X$, if there exist such balls. Otherwise, the join is not defined. 
\end{definition}

Notice that by the equalising property of the atlas $\mathcal{A}$, the notion of join $x\wedge y$ is well-defined for any pair of points $x,y\in X$.
\newline

A function $f\colon X\to\mathds{C}$ is \emph{locally constant}, if it is constant on balls $B\subseteq X$ with sufficiently small radius.
This is equivalent to saying that
for each chart $(U_\alpha,\phi_\alpha)$, the associated function
\[
f_\alpha\colon\phi_\alpha(U)\to\mathds{C}
\]
is constant on sufficiently small balls which get mapped to balls of equal radius under all  the suitable transition functions given by the atlas $\mathcal{A}$.
\newline

In this manner, define
\begin{align*}
\mathcal{D}(X)&=\mathset{f\colon X\to\mathds{C}\mid\text{$f$ is locally constant with compact support}}
\\
C(X)&=\mathset{f\colon X\to \mathds{C}\mid\text{$f$ is continuous}}
\\
L^\rho(X,\mu_X)&=\mathset{f\colon X\to\mathds{C}\mid\forall (U,\phi)\in\mathcal{A}\colon f\in L^\rho(\phi(U),\phi_*\mu_X)}\;(0<\rho\le\infty)\,,
\end{align*}
as function spaces on $X$.
\newline

In particular, the last type of function spaces requires integration on $X$ with respect to the measure $\mu_X$ given by a non-vanishing regular analytic $n$-form $\omega$ on $X$. It can be written out on charts $(U_\alpha,\phi_\alpha),(U_\beta,\phi_\beta)$ as follows: Assume that the integral
\[
\int_{U_{\alpha\beta}}h(x)\,d\mu_X(x)
\]
is to be evaluated on the set
\[
U_{\alpha\beta}=U_\alpha\cap U_\beta\,.
\]
Thus, using diagram (\ref{diagram}), obtain
\begin{align*}
\int_{U_{\alpha\beta}}h(x)\,d\mu_X(x)&=
\int_{\phi_\alpha(U_{\alpha\beta})}h_\alpha(z)\absolute{\omega_\alpha(z)}
=\int_{\tau_{\alpha\beta}^{-1}(\phi_\beta(U_{\alpha\beta}))}
h_\beta(\tau_{\alpha\beta}(y))\absolute{\omega_\beta(\tau_{\alpha\beta}(y))}
\\
&=\int_{\tau_{\alpha\beta}^{-1}(\phi_\beta(U_{\alpha\beta}))}
h_\beta(\tau_{\alpha\beta}(y))\absolute{f_\beta(\tau_{\alpha\beta}(y))}\absolute{d\tau_{\alpha\beta}(y)}
\\
&=\int_{\phi_\beta(U_{\alpha\beta})}
h_\beta(y)\absolute{f_\beta(y)}\absolute{\det\left(\tau_{\alpha\beta}'(y)\right)}\absolute{dy}
\\
&=\int_{\phi_\beta(U_{\alpha\beta})}
h_\beta(y)\absolute{\omega_\beta(y)}\,,
\end{align*}
where we have written out the differential form $\omega$ on $U_{\alpha\beta}$ often encountered in the following sloppy notation: 
\[
\omega_\alpha=\omega|_{U_\alpha}=f_\alpha\,dx,\quad\omega_\beta=\omega|_{U_\beta}=f_\beta\,dx
\]
with analytic functions $f_\alpha\colon U_\alpha\to K$, $f_\beta\colon U_\beta\to K$ as:
\[
\omega_\alpha|_{U_{\alpha\beta}}=f_\alpha\,dx
|_{U_{\alpha\beta}}=f_\beta\,dx|_{U_{\alpha\beta}}=\omega_\beta|_{U_{\alpha\beta}}\,,
\]
but here
explicitly given as
\[
f_\alpha|_{\phi_\alpha(U_{\alpha\beta})}=f_\beta|_{\phi_\beta(U_{\alpha\beta})}\circ\tau_{\alpha\beta}
=\det\left(\tau_{\alpha\beta}'\right)\cdot f_\beta|_{\phi_\beta(U_{\alpha\beta})}
\]
in order to be able to explicitly integrate on the $p$-adic analytic manifold $X$ using charts and transition functions.
\newline

Using this and diagram (\ref{diagram}), it is an exercise to show that the function space $\mathcal{D}(X)$ is dense in the other function spaces defined above, and that the latter are Banach spaces.


\subsection{Local Vladimirov-Taibleson operators}\label{VTop}

A well-defined kernel function 
$k(x,y)$ on the $p$-adic manifold $X\times X$ can be defined using the atlas $\mathcal{A}\times\mathcal{A}$.
Given a kernel function $k(x,y)$ on $X$, one can now define a Laplacian integral operator $\mathcal{K}$ as follows:

\begin{align*}
\mathcal{K}h(x)&=\int_Xk(x,y)(h(x)-h(y))\,\mu_X(y)\,.
\end{align*}

The following definition will make use of the weighted nerve complex $(N(\mathcal{A}),w)$ defined in Section \ref{sec:nerveComplex}.
First, define  for $A\subset X$
\[
\sigma(A)=\text{the highest-dimensional simplex containing 
$A$}\,,
\]
if $A$ is contained in a simplex of $N(\mathcal{A})$.
Use the notation
\[
\sigma(x_1,\dots,x_m)=\sigma(\mathset{x_1,\dots,x_m})
\]
for $x_1,\dots,x_m\in X$, if there exists a simplex in which these points are all contained.
\newline

Define the function
\[
d_g\colon X\times X\to\mathds{R}_{\ge0},\;(x,y)\mapsto
\begin{cases}
\mu_X(x\wedge y),&\text{$x,y$ contained in a ball in $X$}
\\
\min\limits_{\gamma\colon\sigma(x)\leadsto\sigma(y)}\mathset{\mu_X(U_\gamma(x,y))},&
\text{$x\wedge y$ does not exist}
\end{cases}
\]
where $\gamma\colon\sigma(x)\leadsto\sigma(y)$ is a path between $\sigma(x)$ and $\sigma(y)$ in the simplicial complex $N(\mathcal{A})$ passing through a sequence of adjacent faces (of higher or lower dimension), and
\[
U_\gamma(x,y)=\bigcup\limits_{\sigma\in\gamma}\sigma\subseteq X\,,
\]
where the union is taken over the subsets of $X$ corresponding to the simplices along the path $\gamma$.

\begin{Lemma}
The function $d_g$ defines a distance on $X$.
\end{Lemma}

\begin{proof}
This is immediate due to the monoticity of the measure $\mu_X$ w.r.t.\ inclusion.
\end{proof}

\begin{definition}
The function $d_g$ is called the \emph{geodetic distance} on the $p$-adic analytic manifold $(X,\mathcal{A})$.
\end{definition}

\begin{definition}
The operator $\Delta^s$ given by
\begin{align*}
\Delta^s h(x)&=
\int_X d_g(x,y)^{-s}(h(x)-h(y)))\,\absolute{\omega(y)}
\end{align*}
with
for $x\in X$ and $s\in\mathds{R}$,
is called a \emph{$p$-adic Laplace-Beltrami operator}  on $X$, and defines a Laplacian integral operator on $\mathcal{D}(X)$.
\end{definition}

\begin{remark}
The connectedness of the nerve complex $N(\mathcal{A})$ is essential in order to  to be able to transition from anywhere to everywhere in the manifold $(X,\mathcal{A})$.
\end{remark}

An example compact manifold which is not contained inside a $p$-adic ball is the projective $n$-space $\mathds{P}^n(K)$. The main reason is that its nerve complex, depicted in Figure \ref{NerveProjectivePlane}, is a $2$-dimensional simplex, viewed as a simplicial complex. This means that in the case of the $p$-adic projective $n$-plane with the usual atlas, the operator $\Delta^s$ is actually global, because every point is reachable from the maximal $2$-simplex. 
\newline

An example of manifold whose nerve complex is not a simplex, is given by
\[
Y=\mathds{P}^n(K)\setminus U_{0\dots n}=\mathds{P}^n(K)\setminus S_0^n\,,
\]
whose nerve complex $N(\mathcal{A}_Y)$ is the simplicial complex obtained by removing from the $n$-simplex $N(\mathcal{A}_{\mathds{P}^n(K)})$ (as a simplicial complex) the ``interior'' $n$-facet. In $Y$, two points are contained in a common simplex, if and only if they are inside an $n-1$-simplex or in any of its lower-dimensional facets. Notice, that $Y$ is a compact $p$-adic analytic sub-manifold of $(\mathds{P}^n(K),\mathcal{A})$ with atlas $\mathcal{A}_Y$ obtained from $\mathcal{A}$ by restricting the maps $\phi$ in the charts $(U,\phi)$ to the open $U\cap Y$ of $Y$. Notice that the weights on $N(\mathcal{A}_Y)$ are different from those on $N(\mathcal{A})$. 

\subsection{Spectrum of the local Vladimirov-Taibleson Operator}

In order to be able to say something about the $L^2$-spectrum of $\Delta^s$, define now a suitable notion of wavelet. And for this, the notion of \emph{ball} in $X$ is helpful.

\begin{definition}
A \emph{wavelet} on $X$ is a function
\[
\psi\colon X\to\mathds{C}
\]
supported on a ball $B(a)\subseteq X$, with $a\in X$, and for any given chart $(U_\alpha,\phi_\alpha)$ containing $B$, $\psi_\alpha$ is a Kozyrev wavelet supported in a ball  centred in $\phi_\alpha(a)\in K^n$.
\end{definition}

\begin{Lemma}
It holds true that
\[
\int_X\psi(x)\,\absolute{\omega(x)}=0
\]
for any wavelet $\psi$ on $X$.
\end{Lemma}

\begin{proof}
Since this is a well-known fact for Kozrev wavelets in $K^n$, cf.\  \cite[Theorem 3.29]{XKZ2018} or \cite[Theorem 9.4.2]{AXS2010}, the assertion follows immediately.
\end{proof}

\begin{Lemma}\label{waveletEigenvalue}
Any wavelet $\psi$ on $X$ supported in a ball $B\subseteq X$ is an eigenfunction of the $p$-adic Laplacian $\Delta^s$ with eigenvalue
\[
\lambda_\psi = \int_{X\setminus B}d_g(x,y)^{-s}\,\absolute{\omega(y)}+\mu_X(B)^{1-s}\,,
\]
independently of the choice of any $x\in B$, and
where $s\in\mathds{R}$.
\end{Lemma}

\begin{proof}
Independence on the choice of $x\in B$ is immediate.
The proof of  the eigenvalue formula in \cite[Theorem 3]{Kozyrev2004} carries over, because in the decomposition of the integral
\begin{align*}
\int_Xd_g(x,y)^{-s}&(\psi(x)-\psi(y))\,d\mu_x(y)
\\
&=\int_{X\setminus B}d_g(x,y)^{-s}(\psi(x)-\psi(y))\,d\mu_X(y)
\\
&+\int_Bd_g(x,y)^{-s}(\psi(x)-\psi(y))\,d\mu(y)\,,
\end{align*}
it is not needed that $d_g(x,y)$ is an ultrametric, if $x\in B$ and $y\in X\setminus B$. In that proof, it is only used that $d_g(x,y)$ is an ultrametric on the ball $B$.
\end{proof}

Define the function space
\[
L^2(X,\mu_X)_w=\left|
\begin{minipage}{9cm}
closure of the span of wavelets in $X$ with supports so small, that they do not overlap
\end{minipage}
\right|
\]
as a closed subspace of $L^2(X,\mu_X)$. The non-overlapping condition is realisable, because of 
Theorem \ref{finiteEqAtlas}.
The $L^2$-space has as a dense subspace
\[
\mathcal{D}(X)_w=\mathcal{D}(X)\cap L^2(X,\mu_X)_w\,,
\]
and, by Lemma \ref{waveletEigenvalue}, the operator $\Delta^s$ is densely defined on $L^2(X,\mu_X)_w$.
The following result can now be formulated:

\begin{thm}
The Hilbert space $L^2(X,\mu_X)_w$ has an orthonormal basis of eigenfunctions of $\Delta^s$ consisting of wavelets on $X$, where $s\in\mathds{R}$. Each eigenvalue has only finite multiplicity.
\end{thm}

\begin{proof}
The orthonormal  property is given by the smallness of the supports, and by the corresponding property on the image of each chart map plus the fact that the images of balls map to balls via transition map.
The finiteness of the eigenvalue multiplicities follows from Lemma \ref{waveletEigenvalue}.
\end{proof}

The subscript ${}_w$ is going to be used also for other function spaces via intersecting with $\mathcal{D}(X)_w$ or $\mathcal{D}(X,\mathds{R})_w$ which denotes the real-valued functions on $X$ which are locally constant (with compact support).

\begin{Lemma}\label{Feller}
The operator $-\Delta^{s}$ generates a Feller semigroup $e^{-t\Delta^s}$ with $t\ge0$ on $C(X,\mathds{R})_w$ for $s\in\mathds{R}$.
\end{Lemma}

\begin{proof}
The proof follows the lines of the proof of 
 \cite[Lemma 5.1]{SchottkyDiff}, in which we verify the criteria given by the Hille-Ray-Yosida Theorem, cf.\ \cite[Ch.\ 4, Lemma 2.1]{EK1986}

\smallskip
1. The domain of $-\Delta^s$ is dense in $C(X,\mathds{R})_w$. This is verified, as the domain is $\mathcal{D}(X,\mathds{R})_w$ which is dense in $C(X,\mathds{R})_w$.

\smallskip
2. $-\Delta^s$ satisfies the positive maximum principle. Let $h\in \mathcal{D}(X,\mathds{R})_w$, and $x_0\in X$ such that $h$ takes its maximum in $x_0$. By compactness of  $X$, such $x_0\in X$ exists. Then
\[
-\Delta^sh(x_0)\le
\int_X d(x,y)^{-s}(h(x_0)-h(x_0))\absolute{\omega(x)}\le 0\,,
\]
which shows that the positive maximum principle is satisfied.

\smallskip
3. The range $\Range(\eta I+\Delta^s)$ is dense in $C(X,\mathds{R})_w$ for some $\eta>0$. If $s\le0$, then $-\Delta^s$ is a bounded operator, and one can argue as in \cite[Lemma 4.1]{ZunigaNetworks}. In the case $s>0$, the operator is unbounded, and one can now argue as in the proof of \cite[Lemma 5.1]{SchottkyDiff} (by pretending one has the trivial group $\Gamma=1$ acting on $X$). This shows the denseness of the range in $C(X,\mathds{R})_w$.

\smallskip
Since all three criteria of the Hille-Ray-Yosida Theorem are satisfied, the assertion now follows.
\end{proof}

\begin{thm}
There exists a probability measure $p_t(x,\cdot)$ with $t\ge0$, $x\in X$ on the Borel $\sigma$-algebra of $X$ such that the Cauchy problem for the heat equation
\[
\left(\frac{\partial}{\partial t}+\Delta^s\right)u(x,t)=0
\]
having initial condition $u(x,0)=u_0(x)\in C(X,\mathds{R})_w$ has a unique solution in $C^1((0,\infty),X)_w$ of the form
\[
u(x,t)=\int_Xu_0(y)p_t(y,d\mu_X(y))\,.
\]
Additionally, $p_t(x,\cdot)$ is the transition function of a strong Markov process on $X$ whose paths are c\`adl\`ag.
\end{thm}

\begin{proof}
Following the proof of \cite[Theorem 5.2]{SchottkyDiff}, we can argue as follows:

\smallskip
Due to Lemma \ref{Feller}, $-\Delta^s$ is the generator of  a Feller semigroup on $C(X,\mathds{R})_w$. Arguing as in the proof of \cite[Theorem 4.2]{ZunigaNetworks}, there exists a uniformly stochastically continuous $C_0$-transition function $p_t(x,d\mu_X(y))$ satisfying condition (L) of \cite[Theorem 2.10]{Taira2009} such that
\[
e^{-t\Delta^s}h_0(x)=\int_X h_0(y)p_t(x,d\mu_X(y))\,,
\]
cf.\ \cite[2.15]{Taira2009}. Now, from the correspondence between transition functions and Markov processes, it follows that there exists a strong Markov process on $X$,
whose paths are c\`adl\`ag. 
\end{proof}

In fact, there is a heat kernel function associated with the operator $-\Delta^s$:

\begin{thm}\label{heatKernel}
The heat kernel function 
\[
H(t,x,y)=1+\sum\limits_{\psi}e^{-t\lambda_\psi}\psi(x)\overline{\psi(y)}
\]
for $\Delta^s$ with $s\in \mathds{R}$,
where $\psi$ runs over the wavelets on $X$,
exists for $x\neq y
\in X$, $t>0$. If $s>1$, then it also exists for $x=y$ and $t>0$, and
\[
H(t,\cdot,\cdot)\in L^2(X\times X,\mu_X\wedge\mu_X)
\]
in this case.
\end{thm}

\begin{proof}
The convergence of the series $H(t,x,y)$ for $t>0$, and $x\neq y$ follows thus: since $X$ is compact, there are only finitely many wavelets $\psi$ such that 
\[
x,y\in\supp(\psi)\,,
\]
and so $H(t,x,y)$ is a finite sum in this case, hence convergent. This shows the existence of the heat kernel function outside the diagonal of $X\times X$.

\smallskip
If $s>1$, $t>0$, and $x=y$, then
\begin{align*}
H(t,x,x)&=1+\sum\limits_\psi e^{-t\lambda_\psi}\mu_X(\supp(\psi))^{-1}
\\
&\stackrel{(*)}{\le} 1+\sum\limits_\psi
\frac{e^{-\mu_X(\supp(\psi))^{1-s}\,t}}{\mu_X(\supp(\psi))}<\infty\,,
\end{align*}
where $(*)$ follows from Lemma \ref{waveletEigenvalue}. This shows the existence of the heat kernel function for $\Delta^s$ on all of $X\times X$.

\smallskip
For the $L^2$-property, it holds true that
\begin{align*}
\int_{X\times X}&\absolute{H(t,x,y)}^2d\mu_X(x)d\mu_X(y)
\le\mu_X(X)^2+
\sum\limits_\psi e^{-2t\lambda\psi}
<\infty\,,
\end{align*}
because the eigenvalues increase unboundedly with shrinking $\supp(\psi)$ in the case $s>1$, $t>0$.
\end{proof}

The heat kernel function can be written as
\[
H(t,x,y)=1+h(t,x,y)
\]
with 
\[
h(t,x,y)=\sum\limits_\psi e^{-t\lambda_\psi}\psi(x)\overline{\psi(y)}\,,
\]
for $t>0$. The associated Green function is given by
\[
G(x,y)=\int_0^\infty h(t,x,y)\,dt
\]
for $x,y\in X$.

\begin{cor}\label{Green}
The Green function for $-\Delta^s$ with $s>1$ exists, and has the form
\[
G(x,y)=\sum\limits_\psi \lambda_\psi^{-1}\psi(x)\overline{\psi(y)}
\]
for $x,y\in X$.
\end{cor}

\begin{proof}
The form of $G(x,y)$ is given formally by integrating each summand of $h(t,x,y)$. Since
\[
\lambda_\psi\in O\left(\mu_X(\supp(\psi))^{1-s}\right)
\]
for $\mu_X(\supp(\psi))\to 0$,
and
\[
\absolute{\psi(x)\overline{\psi(y)}}
\le\mu_X(\supp(\psi))^{-1}\,,
\]
it follows that
\[
\lambda_\psi^{-1}\absolute{\psi(x)\overline{\psi(y)}}\in O\left(\mu_X(\supp(\psi))^s\right)
\]
with $s>1$. Hence,
\[
\absolute{G(x,y)}\le\sum\limits_\psi
\lambda_\psi^{-1}\absolute{\psi(x)\overline{\psi(y)}}<\infty\,,
\]
and the convergence of $G(x,y)$ follows for $x,y\in X$.
\end{proof}

\begin{remark}
Again, as seen in the proof of Theorem \ref{heatKernel}, due to the compactness of $X$, the  Green function in its form of Corollary \ref{Green} is a finite sum, if $x,y\in X$ are two distinct points.
\end{remark}


\section{Diffusion on an elliptic curve with good reduction}

Let $E$ be an elliptic curve over $K$ with $p\neq 2,3$, and given via its Weierstrass equation:
\[
E\colon y^2=x^3-27c_4x-54c_6\,,
\]
where it w.l.o.g.\ be assumed that $c_4,c_6\in O_K$. The curve is a projective algebraic subvariety of
the projective $n$-space $\mathds{P}^n$ over $K$, and the set $E(K)$ of $K$-rational points form a closed subset of the $p$-adic analytic manifold $\mathds{P}^2(K)$, locally given as the zero set of an algebraic equation in two variables.
Since all points of $E(K)$ are non-singular, it follows that $E(K)$ is a $1$-dimensional closed submanifold of $\mathds{P}^2(K)$. Cf.\ \cite[Chapter 2.4]{IgusaLocalZeta} for the construction of closed submanifolds and an atlas for such.
\newline

Notice that the zero element $O\in E(K)$ in the space of $K$-rational points $E(K)$ has the coordinates
\[
O=[0:1:0]\in\mathds{P}^2(K)
\]
w.r.t.\ to the Weierstrass equation.
\newline

The measure $\mu_E$ on $E(K)$ will be given through the invariant differential $1$-form
\[
\omega=\frac{dx}{2y}\,,
\]
which has no zeros on $E(K)$. It will be written as
\[
\mu_E=\absolute{\omega}=\frac{\absolute{dx}}{\absolute{2y}}\,,
\]
and will play a prominent role in the following subsections.
\newline

The main difference with the case of Tate elliptic curves dealt with in \cite{brad_thetaDiffusionTateCurve}, is that in the case of good reduction, there is no Tate uniformisation available, as explained in \cite{tate1995} or \cite[Chapter V.3]{Silverman1994}, and thus the theory of theta functions  from \cite[Chapter 5.1]{FvP2004} cannot be used. This necessitates a different approach, as will be explained below.

\subsection{$E(K)$ as a $p$-adic analytic manifold}

Assume that $E$ has good reduction. This means that the curve
\begin{align}\label{reducedWeierstrass}
\overline{E}\colon y^2=x^3-\overline{27}\,\overline{c}_4x-\overline{54}\,\overline{c}_6
\end{align}
defined over the residue field $k=O_K/\pi O_K$ is non-singular. In other words, $\overline{E}$ is an elliptic curve over $k$ whose Weierstrass equation is (\ref{reducedWeierstrass}), and whose origin is $\overline{O}=\left[\bar{0}:\bar{1}:\bar{0}\right]\in\mathds{P}^2(k)$.

\begin{thm}\label{SerreGR}
The $K$-rational points $E(K)$ of an elliptic curve $E$ with good reduction with $p\neq 2,3$, form a $1$-dimensional closed $p$-adic analytic submanifold of $\mathds{P}^2(K)$, whose Serre invariant is
\[
i(E(K))=\absolute{\overline{E}(k)}\mod (q-1)\,,
\]
where $q=\absolute{k}$ is the cardinality of the residue field
$k=O_K/\pi O_K$. 
\end{thm}

\begin{proof}
The closed analytic submanifold property is given by the algebraic equations for $E(K)\subset\mathds{P}^2(K)$, cf.\ \cite[Chapter 2.4]{IgusaLocalZeta}.
According to \cite[Theorem 2.2.5]{WeilAAG}, it holds true that
\[
\int_{E(K)}\absolute{\omega}=\frac{1}{q}\absolute{\overline{E}(k)}\,,
\]
since $\omega$ is a nowhere vanishing differential $1$-form on $E(K)$. From  this, the asserted  value of the Serre invariant follows, since
$q\equiv 1\mod q-1$.
\end{proof}

Using the atlas $\mathset{U_0,U_1,U_2}$ of $\mathds{P}^2(K)$, one sees from the fact that
\[
\overline{O}=\left[\bar{0}:\bar{1}:\bar{0}
\right]
\]
is the only non-affine point of $\overline{E}(k)$, that
the set
\[
E(K)\cap (U_0\setminus U_{012})=\emptyset
\]
is void. Hence, a minimal atlas $\mathcal{A}_E$ of $E(K)$ having a connected nerve complex is given by
\[
O_1=E(K)\cap U_1,\quad
O_2=E(K)\setminus E_1
\]
with non-empty overlap $O_1\cap O_2$ whose measure is likely to depend on the particular Weierstrass equation of $E$. In any case, the weighted nerve complex associated with $\mathcal{A}_E$ is
given in Figure \ref{ENC}, where the measures are
\begin{align*}
\mu_E(O_1)&=\mu_1
\\
\mu_E(O_1\cap O_2)&=\mu_{1}-\frac{2}{q}
\\
\mu_E(O_2)&=\frac{2\absolute{\overline{E}(k)}-2}{q}\,,
\end{align*}
and Theorem \ref{SerreGR} has been used.
Notice that the atlas $\mathcal{A}_E$ has the equalising property.
\newline

\begin{figure}[ht]
\[
\xymatrix@=40pt{
\circled{\scriptsize$\mu_1$}
\ar@{-}[r]^{\hspace*{-7mm}\mu_{1}-\frac{1}{q}}
&
\circled{\rule[-4mm]{0pt}{10mm}\scriptsize$\frac{\absolute{\overline{E}(k)}-1}{q}$}
}
\]
\caption{A minimal weighted connected nerve complex for an elliptic curve with good reduction.}\label{ENC}
\end{figure}
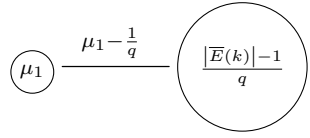

The reduction map
\[
\rho\colon E(K)\to\overline{E}(k)
\]
is, according to \cite[Proposition 2.1]{Silverman1986}, a surjective group homomorphism. This means that the analytic manifold $E(K)$ is the disjoint union of translates of the set
\[
E_1(K):=\ker\rho=\rho^{-1}(\overline{O})\,,
\]
where $\overline{O}\in \overline{E}(k)$ is the origin of the reduced elliptic curve $\overline{E}$. It follows that
\begin{align}\label{measureKernelRed}
\mu_E(E_1(K))=\frac{1}{q}\,,
\end{align}
and $\mu_1$ is an integer multiple of $\frac{1}{q}$.
\newline

Let now $B\subset E(K)$ be a ball, and $y\in E(K)\setminus B$. There are three cases for the value of the geodesic distance
$d_g(B,y)$,
namely:
\newline

\noindent 
\emph{Case $B\subset E_1$.}
Then
\begin{align}\label{dgCase1}
d_g(x,y)=
\begin{cases}
\mu_E(x\wedge y),&x\wedge y\;\text{exists}
\\
\mu_1,&y\in O_1\setminus B_{\max}(x)
\\
\frac{1}{q}\absolute{\overline{E}(k)}\,,&y\in O_2\setminus O_1
\end{cases}
\end{align}
for $x\in B$, and with $B_{\max}(x)$ the largest ball in $E(K)$ containing $x\in B$.
\newline

\noindent
\emph{Case $B\subset O_1\cap O_2$.} Then
\begin{align}
d_g(B,y)=\begin{cases}
\mu_E(x\wedge y),&x\wedge y\;\text{exists}
\\
\mu_1-\frac{1}{q}\,,&y\in O_1\cap O_2\setminus B_{\max}(x)
\\
\frac{1}{q}\left(\absolute{\overline{E}(k)}-1\right)\,,&y\in O_2\setminus O_1\cap O_2
\\
\mu_1\,,&y\in O_1\setminus O_1\cap O_2
\end{cases}
\end{align}
for $x\in B$.
\newline

\noindent
\emph{Case $B\subset O_2\setminus O_1$.} Then
\begin{align}\label{dgCase3}
d_g(x,y)=\begin{cases}
\mu_E(x\wedge y),&x\wedge y\;\text{exists}
\\
\frac{1}{q}\left(\absolute{\overline{E}(k)}-1\right)\,,&y\in O_2\setminus B_{\max}(x)
\\
\frac{1}{q}\absolute{\overline{E}(k)}\,,&y\in O_1\setminus O_2
\end{cases}
\end{align}

\begin{Proposition}\label{GRwaveletEV}
The wavelet eigenvalues of $\Delta^s$ acting on $\mathcal{D}(E(K))$ are
\[
\lambda_\psi=\mu_E(B)^{1-s}+q^{1-s}\left(\mu(B_{\max}(x))^{1-s}-\mu_E(B)^{1-s}\right)+\kappa(s)
\]
for $B=\supp(\psi)\subset E(K)$ with
\[
\kappa(s)=\begin{cases}
a,&B\subset O_1\setminus O_2=E_1(K)
\\
b,&B\subset O_1\cap O_2
\\
c,&B\subset O_2\setminus O_1
\end{cases}
\]
with
\begin{align*}
a&=\mu_1^{-s}\mu_E(O_1\setminus B_{\max}(x))
+\left(\frac{\absolute{\overline{E}(k)}}{q}\right)^{1-s}
\\
b&=\left(\mu_1-\frac{1}{q}\right)^{-s}\mu_E\left(O_1\cap O_2\setminus B_{\max}(x)\right)
+\left(\frac{\absolute{\overline{E}(k)}-1}{q}\right)^{1-s}
+
\mu_1^{1-s}
\\
c&=\left(\frac{\absolute{\overline{E}(k)}-1}{q}\right)^{-s}\mu_E\left(O_2\setminus B_{\max}(x)\right)
+\left(\frac{\absolute{\overline{E}(k)}}{q}\right)^{1-s}\,,
\end{align*}
and with $x\in B$, $s\in\mathds{R}$.
\end{Proposition}

\begin{proof}
This follows from Lemma \ref{waveletEigenvalue} under the consideration of equalities (\ref{dgCase1})-(\ref{dgCase3}).
\end{proof}

\subsection{Hearing the number of points on $\overline{E}(k)$}

Here, the case of an elliptic curve $E$ with good reduction will be studied over a local field $K$ with $p\neq 2,3$ and residue field $k$ of cardinality $q=p^f$ with $f>>0$. The number $q$ will be considered sufficiently large, if the reduction curve $\overline{E}(k)$ under the $2$-sheeted cover
\[
\overline{\lambda}\colon \overline{E}(k)\to\mathds{P}^1(k),\;
(x,y)\mapsto x\,,
\]
contains all the branch points. By the Riemann-Hurwitz formula, the cover $\overline{\lambda}$ has four ramification points: one of them 
is $\overline{O}$, and the other ones are the pre-images of the zeros of the polynomial $\bar{f}$ from the proof of Theorem \ref{SerreGR}. This means that the value for $\mu_1$ of the previous subsection can now be explicitly given as the 
 pre-image under $\rho$ of the unramified locus in $\overline{E}(k)$ w.r.t.\ the cover $\bar{\lambda}$. This is the open set $O_1\cap O_2$. Hence, the following measures hold true:
\begin{align*}
\mu_E(O_1\setminus O_2)=\frac{1}{q}\,,\quad
\mu_E(O_2\setminus O_1)=\frac{3}{q}\,,\quad
\mu_E(O_1\cap O_2)=\frac{\absolute{\overline{E}(k)}-4}{q}\,,
\end{align*}
and thus
\begin{align}\label{mu1}
\mu_1=\frac{\absolute{\overline{E}(k)}-3}{q}\,,\quad\mu_E(B_{\max}(x))=\frac{1}{q}
\end{align}
for any $x\in E(K)$. The latter follows by  translating the disc $E_1(K)$ within $E(K)$. That $E_1(K)$ is indeed an analytic disc follows from the proof of \cite[Theorem 2.2.5]{WeilAAG}, where the fibres of the reduction map of $n$-dimensional algebraic varieties over $O_K$ with good reduction are analytically isomorphic to $\pi O_K^n$.

\begin{Cor}\label{EVGRfbig}
The wavelet eigenvalues of $\Delta^s$ for an elliptic curve with good reduction with $q=p^f$ with $f>>0$ and $p\neq 2,3$ are
\[
\lambda_\psi(s)=1+\left(1-q^{1-s}\right)\mu_E(B)^{1-s}+\kappa(s)
\]
with 
\[
\kappa(s)=\begin{cases}
a(s),&B\subset E_1(K)
\\
b(s),&B\subset O_1\cap O_2
\\
c(s),&B\subset O_2\setminus O_1
\end{cases}
\]
and  $\psi$ a wavelet supported in $B\subset E(K)$, and
\begin{align*}
a(s)&=\left(\frac{\absolute{\overline{E}(k)}-3}{q}\right)^{-s}
\left(\frac{\absolute{\overline{E}(k)}-4}{q}\right)+\left(\frac{\absolute{\overline{E}(k)}}{q}\right)^{1-s}
\\
b(s)&=\left(\frac{\absolute{\overline{E}(k)}-4}{q}\right)^{-s}\left(\frac{\absolute{\overline{E}(k)}-5}{q}\right)+\left(\frac{\absolute{\overline{E}(k)}-1}{q}\right)^{1-s}+\left(\frac{\absolute{\overline{E}(k)}-3}{q}\right)^{1-s}
\\
c(s)&=\left(\frac{\absolute{\overline{E}(k)}-1}{q}\right)^{-s}\left(\frac{\absolute{\overline{E}(k)}-2}{q}\right)+\left(\frac{\absolute{\overline{E}(k)}}{q}\right)^{1-s}
\end{align*}
for $s\in\mathds{R}$.
\end{Cor}

\begin{proof}
This is an immediate consequence of Proposition \ref{GRwaveletEV} using (\ref{mu1}).
\end{proof}

In order to be able to hear the number of points of $\overline{E}(k)$, use Corollary \ref{EVGRfbig} in order to extract this number from 
\[
\lambda_0(s)=\inf\mathset{\lambda_\psi(s)\mid\psi\;\text{wavelet eigenfunction of $\Delta^s$}}
\]
for $s\in\mathds{R}$.

\begin{thm}
Let $E$ be an elliptic curve over $K$ with $q=p^f$ and $f>>0$, $p\neq 2,3$. Then the number $\absolute{\overline{E}(k)}$ satisfies the following property:
\begin{align*}
\absolute{\overline{E}(k)}&=
\begin{cases}
\left(\frac{\displaystyle \raisebox{1mm}{$3$}}{\displaystyle\rule{0pt}{4mm}\lambda_0(s)-q^{s-1}}\right)^{\frac{1}{s-1}}+t_0+q\,,&\text{for some $t_0\in(0,5)$, if $s>>1$},
\\[5mm]
\frac{\displaystyle 6-4\lambda_0(1)}{\displaystyle 3-\lambda_0(1)}\,,&\text{if $s=1$},
\\[5mm]
\left(\frac{\displaystyle \raisebox{1mm}{$2$}}{\displaystyle\lambda_0(s)-1}\right)^{\frac{1}{s-1}}+t_0+q\,,&\text{for some $t_0\in(1,5)$ if $s<1$}\,,
\end{cases}
\end{align*}
and
\[
\left(\frac{2}{\lambda_0(s)-q^{s-1}}\right)^{\frac{1}{s-1}}+q
<
\absolute{\overline{E}(k)}
<
\left(\frac{3}{\lambda_0(s)-q^{s-1}}\right)^{\frac{1}{s-1}}+5+q
\]
if $s>1$.
\end{thm}

\begin{proof}
\emph{Case $s>1$.}
In this case, $\lambda_\psi(s)$ attains its minimum when $B=\supp\psi$ is maximal, i.e.\ $\mu(B)=\mu(B_{\max})=q^{-1}$. This implies
\[
\lambda_0(s)=q^{s-1}+m(s)\,,
\]
where 
\[
m(s)=\min\mathset{a(s),b(s),c(s)}
\]
with $s\in\mathds{R}$. Since it is assumed that $f>>0$, it can be assumed by the Hasse bound \cite[Theorem IV.9.4]{Milne2021}  (Riemann hypothesis for elliptic curves over finite fields) that
\[
\frac{\absolute{\overline{E}(k)}-t}{q}<\min\mathset{1,\frac{\absolute{\overline{E}(k)}-1}{q}}
\]
for $t>1$. It follows 
that
\[
2\left(\frac{\absolute{\overline{E}(k)}}{q}\right)^{1-s}<m(s)
<3\left(\frac{\absolute{\overline{E}(k)}-5}{q}\right)^{1-s}\,,
\]
which for $s>>1$ can be tightend to
\[
3\left(\frac{\absolute{\overline{E}(k)}}{q}\right)^{1-s}<m(s)
<3\left(\frac{\absolute{\overline{E}(k)}-5}{q}\right)^{1-s}\,,
\]
and thus
\[
m(s)=3\left.\left(\frac{\absolute{\overline{E}(k)}-t}{q}\right)^{1-s}\right|_{t=t_0}
\]
for some $t_0\in(0,5)$ by continuity of that function in the variable $t\in[0,5]$. It follows that
\begin{align*} 
\absolute{\overline{E}(k)}=\left(
\frac{3}{\lambda_0(s)-q^{s-1}}\right)^{\frac{1}{s-1}}+t_0+q
\end{align*}
in the case $s>>1$. In the general case of $s>1$, it follows that
\begin{align*} 
\left(\frac{2}{\lambda_0(s)-q^{s-1}}\right)^{\frac{1}{s-1}}+q
<
\absolute{\overline{E}(k)}
<
\left(\frac{3}{\lambda_0(s)-q^{s-1}}\right)^{\frac{1}{s-1}}+5+q
\end{align*}

\noindent
\emph{Case $s=1$.} If $s=1$, then
\[
m(1)=1+\frac{\absolute{\overline{E}(k)}-5}{\absolute{\overline{E}(k)}-4}
=\lambda_0(1)-1\,,
\]
from which it follows that
\begin{align*} 
\absolute{\overline{E}(k)}=
\frac{6-4\lambda_0(1)}{3-\lambda_0(1)}
\end{align*}
with $f>>0$.
\newline

\noindent
\emph{Case $s<1$.} If $s<1$, then 
\[
m(s)=b(s)=\lambda_0(s)-1,
\]
and thus
\[
3\left(\frac{\absolute{\overline{E}(k)}-5}{q}\right)^{1-s}
< m(s)
< 3\left(\frac{\absolute{\overline{E}(k)}-1}{q}\right)^{1-s}\,,
\]
which implies, similarly as in the case $s>>1$, that
\begin{align*} 
\absolute{\overline{E}(k)}=\left(\frac{3}{\lambda_0(s)-1}\right)^{\frac{1}{s-1}}+t_0+q
\end{align*}
with $t_0\in(1,5)$. This proves the assertions.
\end{proof}

\section*{Acknowledgements}
The author wants to thank
Wilson Z\'u\~{n}iga-Galindo for 
asking questions which lead to this research. David Weisbart and \'Angel Mor\'an Ledezma are thanked for fruitful discussions.
Frank Herrlich 
is thanked for my formation in algebraic geometry many years ago.

\section*{Statements and Declarations}

\paragraph{Funding.}
This research is partially funded by the Deutsche Forschungsgemeinschaft under 
project number 469999674.

\paragraph{Competing Interests.} The author has  no relevant financial or non-financial interests to disclose.

\paragraph{Author Contributions.}
The author contributed to the study conception and design,
and analysis. 
The author wrote, read and approved the final manuscript.

\section*{Data availability}

There is no data associated with this manuscript.

\bibliographystyle{plain}
\bibliography{biblio}

\end{document}